\documentclass[10pt, %reqno
]{amsart}
\usepackage{amssymb,mathrsfs}
\usepackage{amssymb}
\usepackage{amsmath}
\usepackage{amsfonts, mathrsfs}
\usepackage{epsfig}
\usepackage{color}
\usepackage{epstopdf}
\usepackage{graphicx}
\usepackage{srcltx}
\usepackage{amsthm}
\usepackage{enumerate}
\usepackage[mathscr]{eucal}
\usepackage{verbatim}
\usepackage[draft]{todonotes}
\usepackage{relsize}
\usepackage[%hidelinks, 
]
{hyperref}

\usepackage[normalem]{ulem}

\hypersetup{
	colorlinks=false,       % false: boxed links; true: colored links
	linkcolor=blue,          % color of internal links (change box color with linkbordercolor)           
	citecolor=red,        % color of links to bibliography
	filecolor=magenta,      % color of file links
	urlcolor=cyan           % color of external links
}

\newcommand{\Be}{\begin{equation}}
\newcommand{\Ee}{\end{equation}}
\newcommand{\Bea}{\begin{eqnarray}}
\newcommand{\Eea}{\end{eqnarray}}
\newcommand{\Bel}{\begin{align}}
\newcommand{\Eel}{\end{align}}
\newcommand{\Beas}{\begin{eqnarray*}}
	\newcommand{\Eeas}{\end{eqnarray*}}
\newcommand{\Benu}{\begin{enumerate}}
	\newcommand{\Eenu}{\end{enumerate}}
\newcommand{\Bi}{\begin{itemize}}
	\newcommand{\Ei}{\end{itemize}}

\newcommand{\B}{\Big}

\numberwithin{equation}{section}
\newcommand{\supp} {\text{supp\! }}

\theoremstyle{plain}
\newtheorem{thm}{Theorem}[section]

\newtheorem{lem}[thm]{Lemma}
\newtheorem{prop}[thm]{Proposition}

\theoremstyle{remark}
\newtheorem{rmk}{Remark}

\theoremstyle{definition}
\newtheorem{defn}{Definition}[section]

\newcommand{\spn}[1]{\text{ span}\Big\{#1\Big\}} \newcommand{\spnn}[1]{\text{ span}\{#1\}}
\newcommand{\mb}{\mathbf}

\begin{document}

 \title[Restriction estimates over the sphere]
{
Remarks on estimates for \\the  adjoint restriction operator to curves  \\ over the sphere
}

\author{Seheon Ham}
\author{Hyerim Ko}
\author{Sanghyuk Lee}

%\address{Research institute of Mathematics, Seoul National University, Seoul 08826, Republic of Korea}

\address{Department of Mathematical Sciences and RIM, Seoul National University, Seoul 08826, Republic of Korea}

\email{seheonham@snu.ac.kr}
\email{kohr@snu.ac.kr} 
\email{shklee@snu.ac.kr} 

%\thanks{}
\subjclass[2010]{42B10}

%\date{\today}

\keywords{Fourier restriction theorem, finite type curves}
% ----------------------------------------------------------------

\begin{abstract}
Recently, two of the authors obtained estimates for the adjoint restriction operator to finite type curves with respect to general measures. Strikingly, it turns out that some of such estimates are sharp, especially when the measures are given by surface measures under certain condition. A typical example is the surface measure on the sphere.  We demonstrate sharpness of such estimates by constructing an example and, also, discuss related estimates over different type of surfaces. 
\end{abstract}

\maketitle

\section{Introduction}
Let $\gamma : I=[0,1] \to \mathbb R^d$ be a smooth curve,  and let  the operator $T_\lambda^\gamma$ be 
defined by
\begin{equation}\label{Tf}
T_\lambda^\gamma f(x) = \int_I e^{i \lambda x \cdot \gamma(t)} f(t) \,dt.
\end{equation}
The problem of characterizing $p,q$ for which 
\begin{equation}
	\label{rest} 
\|T_\lambda^\gamma f\|_{L^q(\mathbb R^d)} \le C \lambda^{-\frac{d}{q}} \| f\|_{L^p(I)}
\end{equation}
holds has been studied by many authors and  the estimates have  been established up to the optimal range 
for a large class of curves. 
In particular, for the curves which satisfy the nonvanishing torsion condition 
\begin{equation}\label{nonv}
\det (\gamma'(t), \dots, \gamma^{(d)}(t)) \neq 0 \quad \mbox{on}
\quad I,
\end{equation}
the estimates on the optimal range were obtained by Zygmund \cite{Zygmund} and Drury \cite{Drury}
(also see \cite{Fefferman, Prestini, Christ}) and	
generalized to the variable curve cases by \cite{Hormander, BakLee, BOS09}.
When the curves are degenerate, 
instead of the Lebesgue measure $dt$ the affine arclength measure 
is used to recover the estimate in the optimal range.  For more details regarding the restriction problems for the curves we refer the readers to  \cite{Sjolin, DM85, DM87, Oberlin, BOS08, BOS09, BOS13,  DendrinosWright, DendrinosMuller, Stovall, Chen3} and the references therein.

In this note, we are concerned with the estimate for  $T_\lambda^\gamma f$ over the sphere instead of $\mathbb R^d$. To be more precise, 
we consider the estimate 
\begin{equation}\label{pq}
\|T_\lambda^\gamma f\|_{L^q(\mathbb S^{d-1})} \le C \lambda^{-\frac{d-1}{q}} \| f\|_{L^p(I)}
\end{equation}
and 
we investigate the optimal range of $p$ and $q$ for which  the estimate \eqref{pq} holds.
The bound 
$\lambda^{-\frac{d-1}{q}}$ is the best possible one can expect (see Remark \ref{optimal}). 
By rescaling the estimate \eqref{pq} is  equivalent to $\|T_1^\gamma f\|_{L^q(\lambda \mathbb S^{d-1})} \le C\| f\|_{L^p(I)}$. 

The estimate \eqref{pq} can generally be regarded as an estimate for a special case of degenerate oscillatory integral operators (see \cite{BennettSeeger}). When $d=2$, Greenleaf and Seeger \cite{GreenleafSeeger} proved that \eqref{pq} holds
if and only if $q \ge 3$ and $1/p+2/q \le 1$.  
The argument in \cite{GreenleafSeeger} is based on kernel estimates for the
oscillatory integral operators with the folding canonical relation.
Also, Bennett, Carbery, Soria, and Vargas \cite{Bennett4} obtained the
same result  via the weighted $L^2$ inequality  for  the Fourier
extension operator defined by the circle. 
We further remark that  Bennett and Seeger \cite{BennettSeeger} obtained  {the optimal $p,q$ range of the} $L^p(\mathbb S^2)-L^q(\lambda \mathbb S^2)$ estimates for $\widehat{fd\sigma}$ 
with the spherical measure $ \sigma$.

The following is our first result which gives the sharp $p,q$ range for the estimate \eqref{pq}.

\begin{thm}\label{necessary2}
Let $d \ge 2$. If  $\gamma$ satisfies \eqref{nonv},
then
\eqref{pq} holds provided that 
\begin{equation} \label{pqr}
q> (d^2+d)/2,   \ \  1/p+(d^2+d-2)/2q <1.
\end{equation}
The result  is sharp in that \eqref{pq}  fails if either 
\begin{equation}\label{qn}
q< (d^2+d)/2,
\end{equation}
or
\begin{equation}\label{pqn}
1/p+(d^2+d-2)/2q >1.
\end{equation}
\end{thm}

As is to be seen in its proof, the necessity part of Theorem \ref{necessary2}  
remains valid with $\mathbb S^{d-1}$ replaced by 
any compact smooth hypersurface $S$
as long as a tangent vector of $\gamma$ is parallel to a normal vector to $S$ at a point where the Gaussian curvature  is nonvanishing.  

For $d=2$, Theorem \ref{necessary2} verifies again that sharpness of the aforementioned estimate by Greenleaf and Seeger \cite{GreenleafSeeger} (as well as that in \cite{Bennett4}).  %shows that the aforementioned estimate by Greenleaf and Seeger \cite{GreenleafSeeger} (as well as that in \cite{Bennett4})can not be extended to wider range.
It is likely that the estimate continues to be true for the critical case  $1/p+(d^2+d-2)/2q =1$ or $q= (d^2+d)/2$.
On the other hand, it should be mentioned that  \eqref{rest}  does not hold at the endpoint $q=(d^2+d+2)/2$ for nondegenerate curves.
The failure can be shown by making use of the result  in Arkhipov, Chubarikov, and Karatsuba \cite{Arkhipov3} (also see \cite{Mockenhaupt}) when $\gamma$ is the moment curve, and for the general nondegenerate curve $\gamma$ it was shown by Ikromov \cite{Ikromov}.  
But the weak type version of estimate \eqref{rest} was established at the endpoint case $p=q=({d^2+d+2})/{2}$ by Bak, Oberlin and Seeger \cite{BOS09} for $d \ge 3$, while it fails for $d=2$  as was shown by  Beckner, Carbery, Semmes and Soria \cite{Beckner4}.  

\begin{rmk} \label{beckner} 
In fact,  it was shown in \cite{Beckner4} that the $L^{p,1}(\mathbb S^{d-1}) - L^{q,\infty}(\mathbb R^d)$ estimate for the extension operator $f \mapsto \widehat{fd\sigma}$ does not hold for $p=q=2d/(d-1)$, but 
without difficulty their argument can be modified to  show
the failure of  even the weaker $L^{p,1}(\mathbb S^{d-1})- L^{2d/(d-1),\infty}(\mathbb R^d)$ estimate for any $p>2d/(d-1)$. 
We provide a proof of this in Section 4. 
\end{rmk}

The estimate for the restriction of $T_\lambda^\gamma  f$ to the sphere $\mathbb S^{d-1}$ was earlier  studied by Brandolini, Gigante, Greenleaf, Iosevich, Seeger, and  Travaglini \cite{Brandolini}  but they considered simpler input function $\chi_I$ instead of general $f$, and they obtained the sharp decay rate of the Fourier transform of measures supported on curves.  
By contrast, Theorem \ref{necessary2} provides  the maximal decay rate $(d-1)/{q}$ for general $f\in L^p$.

As mentioned in the  above, for $d=2$ the optimal result including the end line cases was obtained in \cite{GreenleafSeeger} and \cite{Bennett4}. For $d\ge 3$, the sufficiency part of Theorem \ref{necessary2}  follows from  Theorem \ref{HL} below which is a special case of  \cite[Theorem 1.1]{HamLee}.  
In \cite{HamLee}, the estimates with respect to general $\alpha$--dimensional measure (see Definition \ref{dimensional}) were obtained and those results  are sharp in that there are $\alpha$--dimensional measures for which the estimate fails outside of the asserted region. Clearly, since the surface measure is $(d-1)$--dimensional,  from  Theorem \ref{hamlee11} (\cite[Theorem 1.1]{HamLee} with $\alpha=d-1$)  
we immediately have the following.

\begin{thm} 
\label{HL}
	Let $d \ge 3$ and let $S$ be a compact smooth hypersurface in $\mathbb R^d$. For $\gamma$ satisfying \eqref{nonv},
	there exists $C>0$ such that 
$
\|T_\lambda^\gamma f\|_{L^q(S)} \le C \lambda^{-(d-1)/q} \| f\|_{L^p(I)}$
	holds if
	$	q > (d^2+d)/2$ and $
	1/p+( d^2+d -2) /2q < 1.$
\end{thm}

It is rather surprising that Theorem \ref{HL} gives the sharp results since the result in \cite{HamLee} does not rely on specific geometric properties of the associated measures  but only on  the dimensional condition of the measure.  Thus our main contribution here is to show the failure of the estimate \eqref{pq} for the cases \eqref{qn} or \eqref{pqn}. 
Necessity part of Theorem \ref{necessary2} can be generalized to the oscillatory integral operator  $\mathfrak T_\lambda$  defined by
\begin{equation}\label{Tf'}
\mathfrak T_\lambda f(y) = \int_I e^{i \lambda \Psi( y,t) } a(y,t) f(t) \,dt,   
\end{equation}
where $ a\in C_0^\infty(\mathbb R^{d-1}\times { \mathbb R})$ is supported in a neighborhood of the origin and
$\Psi$ is a smooth real-valued function on the support of $a$.

\begin{prop}\label{variable2}
For $d \ge 2$ let $\mathfrak T_\lambda$ is given by \eqref{Tf'}. Suppose that  $\partial_t \nabla_{y} \Psi(0,0) =0$,   and suppose that,  for $(y,t)$ contained in the support of $ a$, 
\begin{equation}\label{M'}
\det \left( 
\partial_t^2 \nabla_{y} \Psi (  {y}, t), \dots,
\partial_t^{d} \nabla_{y} \Psi (  {y}, t)
\right) \neq 0,
\end{equation}
 and
\begin{equation}\label{Gcurvature}
\det(\nabla_{y} \partial_t \nabla_{y} \Psi (y, t))\neq0. 
\end{equation}
Then the estimate $\| \mathfrak T_\lambda f\|_{L^q(\mathbb R^{d-1})} \le C\lambda^{-\frac{d-1}q}\|f\|_{L^p(I)}$ fails  if  either \eqref{qn}  or \eqref{pqn} holds.
\end{prop}

\newcommand{\mba}{\mathbf a}
 
Hence, application of Proposition  \ref{variable2}  to the setting  of  Theorem \ref{necessary2} shows the necessity part of  Theorem \ref{necessary2} (see Section \ref{proof13}). It is plausible to expect that  the estimate  $\| \mathfrak T_\lambda f\|_q \le C\lambda^{-\frac{d-1}q}\|f\|_p$ is true up to 
the critical cases $q = (d^2+d)/2$ and $1/p+(d^2+d-2)/2q = 1$. 
However at the time of this writing, we do not know whether this is true or not.

\noindent{\it Finite type curves.}  Let us set $\mathcal A=\mathcal A(d)=\{\mathbf a = (\mathbf a_1,\dots, \mathbf a_d):   \mathbf a_i\in \mathbb N, \ i=1, \dots, d, \ 1 \le \mathbf a_1 < \cdots < \mathbf a_d\}$
and $\| \mathbf a\|_1 = \mathbf a_1 +\dots +\mathbf a_d$.
We recall the  following from \cite[Definition 1.2]{HamLee} (also see \cite{Christ}).

\begin{defn} Let $\gamma:I=[0,1] \to \mathbb {R}^d,\, d\ge 2 $ be a
smooth curve.
 We say that $\gamma$ is of finite
type at $t\in I$ if there exists $\mba=\mathbf a(t)\in \mathcal A$ such that
\begin{equation}  
\label{deter}
\det \begin{bmatrix} \gamma^{(\mathbf a_1)}(t),& \gamma^{(\mathbf a_2)}(t),
&\cdots,& \gamma^{(\mathbf a_d)}(t)
\end{bmatrix}\neq 0.
\end{equation}
Here the column vectors $\gamma^{(\mathbf a_i)}(t)$ are $\mathbf a_i$--th
derivatives of $\gamma$. 
We say $\gamma$ is of type $\mathbf b\in \mathcal A$ at $t$ if 
the minimum of $\|\mathbf a(t)\|_1$ over all the possible choices of $\mathbf a(t)$ for  which
\eqref{deter} holds  is attained when $\mathbf a(t)=\mathbf b$. 
We also say that $\gamma$ is of finite type if
so is $\gamma$ at every $t\in I$. 
\end{defn}

\begin{thm}\label{rmkk2}
Let $d \ge 3$ and $\gamma$ be of finite type. Suppose that  $\gamma$ is of type $\mathbf a(t)$ at $t$ 
and   $ \|\mathbf a(t_0)\|_1- \frac{d^2+d}2 \ge 1$ for some $t_0\in I$.  
Then, for $p,q$ satisfying $q >\frac{d^2+d}2$ and $1/p +  \max_{t \in I} \{  \|\mathbf a(t)\|_1-\mathbf a_1(t)\} /q \le 1$, 
\begin{equation}
\label{fnt}
\|T_\lambda^\gamma f\|_{L^{q}(\mathbb S^{d-1})} \le C \lambda^{-\frac{d-1}{q}} \| f\|_{L^{p,q}(I)}
\end{equation}
holds. 		%
Furthermore \eqref{fnt} fails if $1/p+  \max_{t \in I} \{  \|\mathbf a(t)\|_1-\mathbf a_1(t)\}/q > 1$. 
\end{thm}

Note that $\|\mathbf a\|_1\ge \frac{d^2+d}2$ if $\mathbf a\in \mathcal A$. Thus, if $\gamma$ dose not satisfy the assumption of Theorem \ref{rmkk2},  $\|\mathbf b(t)\|_1= \frac{d^2+d}2$  for all $t\in I$.  
This case was already considered in Theorem \ref{HL}. 
Note that for $q \ge p$, \eqref{fnt} implies the strong type $(p,q)$  estimate for $p,q$ which satisfy  $ \frac 1p + \max_{t \in I} \{  \|\mathbf b(t)\|_1-b_1(t)\}/q \le 1 $ by the inclusion $L^p \subset L^{p,q}$.  In the case of $q<p$, the strong type estimate for $ \frac 1p + (\|\mathbf b\|_1-b_1)\frac1q< 1$ follows by  H\"{o}lder's inequality in the Lorentz space.

Theorem \ref{rmkk2} is to be shown by considering the finite type curve as a union of small perturbation of monomial curves, which can be normalized into the curves contained in  $\mathfrak G^{\mathbf a}(\epsilon)$ (see \eqref{finiteclass}).
%
%Though the curves in $\mathfrak G^{\mathbf a}(\epsilon)$ are degenerate at the origin, they are not degenerate away from the zero if $\epsilon>0$ is small enough. 
If $\epsilon$ is small enough, the torsion of curves in $\mathfrak G^{\mathbf a}(\epsilon)$ can be controlled uniformly and vanish only at the origin.
 By dyadic decomposition away from the origin, 
%This can be exploited by dyadic decomposition away from the origin. In fact, 
we can apply Theorem \ref{hamlee11} for the curves on each dyadic interval via rescaling. 
For the purpose we will consider $L^p - L^q$ estimate for $T_\lambda^\gamma$ with respect to general $\alpha$--dimensional measures, which was considered in \cite{HamLee}
(also see \cite{Mockenhaupt2, Mitsis, BakSeeger} for the Stein-Tomas restriction theorem with respect to general measures).

\subsubsection*{Hyperplane} As is to be seen later, in Theorem \ref{necessary2}, i.e. the case of $\mathbb S^{d-1}$, the sharpness of the range of $p,q$ is shown by making use of the  fact that  for any tangent vector $\gamma'$ to $\gamma$ there is a normal  vector to the sphere which is parallel to  $\gamma'$.  However, this is not the case for  hyperplanes, so it is natural to expect that a weak type version of \eqref{pq} generically holds on a  wider range of $p,q$ than that in Theorem \ref{necessary2}. 
In the following we provide a complete characterization of $p,q$ for which 
a weak type version of \eqref{pq} holds.

\begin{prop}\label{hyper}
Let $d \ge 3$,  $\gamma:I \to \mathbb R^d$ be of finite type, and let $H$ be a hyperplane with a normal vector 
	$\mathbf n$.  
	Suppose that $\gamma$ is of type $\mathbf b(t) =(\mathbf b_1(t),\dots, \mathbf b_d(t))\in \mathcal A(d)$ at each $t \in I$. Let $\omega(t)$ be the minimum of $\mathbf a_1+\dots + \mathbf a_{d-1}$ while $\mathbf a_i \in 
\{\mathbf b_1(t),\dots, \mathbf b_d(t) \}$ and the vectors $ \mathbf n,  \gamma^{ (\mathbf a_1)}(t) ,\dots,\gamma^{ (\mathbf a _{d-1})}(t)
$ are linearly independent and set 
\[ \omega_\ast= \max_{t\in I} \omega(t).\] 
Then,     	the estimate
	\begin{equation}\label{TH2}
	\| T_\lambda^\gamma f\|_{L^q(H)} \lesssim \lambda^{-(d-1)/q} \|f\|_{L^{p,q}}
	\end{equation}
	holds if and only if $q > d(d-1)/2 +1$ and  $1/p +     \omega_\ast /q \le 1$.
\end{prop}

The necessity of the condition $1/p +     \omega_\ast /q \le 1$
can be shown by using a Knapp type example (for example, see the proof of  Proposition \ref{AA}).
When $q< p$, the failure of the estimate $\| T_\lambda^\gamma f\|_{L^q(H)} \lesssim \lambda^{-(d-1)/q} \|f\|_{L^{p}}$  with the critical $p,q$ satisfying  $1/p +     \omega_\ast  /q = 1$
was shown in  \cite[Section 5]{Stovall}.  

In order to compare Proposition \ref{hyper} with  Theorem \ref{necessary2}, we consider the case of a nondegenerate curve $\gamma$. In this case $\omega_\ast$ takes its value in $[d(d-1)/2, d(d-1)/2 +1]$. 
If  $\omega_\ast= d(d+1)/2 -1$,
the range of $p,q$ in Proposition \ref{hyper} becomes the smallest
but it properly contains the range $p,q$  in Theorem \ref{necessary2}. 
So \eqref{TH2} holds  for $p,q$ which are contained in a wider range than that of \eqref{pqr}.  
This  explains how the curvature of the surface plays a significant role even in the nondegenerate case.
On the other hand, if $\omega_\ast= d(d-1)/2$,
we get the largest range of $p,q$ which coincides with that of the adjoint restriction estimate to the nondegenerate curves in $\mathbb R^{d-1}$.

\begin{rmk}\label{k-dim}
The result in \cite{HamLee} (Theorem \ref{hamlee11}) also shows that $ \|T_\lambda^\gamma f\|_{L^q(S)} \le C \lambda^{-\frac{k}{q}} \| f\|_{L^p(I)}$ holds for any $k$--dimensional compact submanifold  $S$ for $k\ge 2$ whenever $1/p+(2d-k+1)k/2q < 1$ and $q > (2d-k+1)k/2+1$. 
In Section 4, we show that  the condition $q \ge (2d-k+1)k/2+1$ is  generally necessary  by constructing a $k$--dimensional submanifold $S$ for which $ \|T_\lambda^\gamma f\|_{L^q(S)} \le C \lambda^{-\frac{k}{q}} \| f\|_{L^p(I)}$ fails if  $q < (2d-k+1)k/2+1$.
\end{rmk}

\begin{rmk}\label{optimal}
The decay rate $\lambda^{-(d-1)/q}$ in \eqref{pq} is optimal
for any smooth hypersurface $S$.
We consider
a ball $B(x_0,\lambda^{-1})$ such that
$|S \cap B(x_0,\lambda^{-1})| > C\lambda^{-(d-1)}$.
Let us take $f(t) = \chi_{[0,\epsilon_0]}(t)e^{-i\lambda x_0\cdot \gamma(t)}$.
With a small enough $\epsilon_0>0$, $| T_\lambda^\gamma f (x)| \gtrsim 1$ if $x \in B(x_0, \lambda^{-1})$.
Thus we see $
\| T_\lambda^\gamma f \|_{L^q(S\cap B(x_0,\lambda^{-1}))}
\ge C \lambda^{-(d-1)/q}.
$ This shows the optimality of the bound.
\end{rmk}

\emph{Outline of the paper.}
In Section 2, we make  observations regarding  geometric properties of
the phase function, and  
we prove Theorem \ref{necessary2} and Proposition \ref{variable2}
by  randomization argument based on  Khintchine's inequality 
and by adapting the Knapp type example.    
The proofs of Theorem \ref{rmkk2} and Proposition \ref{hyper} are given in Section 3. 
In Section 4 we provide details concerning  Remark \ref{beckner} and  the example mentioned in Remark  \ref{k-dim}.  

Finally, for $A, B > 0$ we write $A\lesssim B$ if $A \le C B$ for a constant $C$. Also the constant $C$ may differ at each occurrence.

\section{Proof of Theorem \ref{necessary2} and Proposition \ref{variable2}}
\label{proof13}

We first prove Proposition \ref{variable2} by using a randomization argument for \eqref{qn} and modifying  the Knapp example for \eqref{pqn}. 
Then, we use   Proposition \ref{variable2}  to show the necessity part of  Theorem \ref{necessary2}.

\subsection{Proof of Proposition \ref{variable2}}
Since $\partial_t \nabla_y \Psi(0,0) = 0$ and $\det \nabla_y\partial_t \nabla_y \Psi \neq 0$ on the support of $a$, 
by the implicit function theorem,  there exists a neighborhood $U\times V \subset \mathbb R^{d-1}\times \mathbb R$ of $(0,0)$ and a $C^1$ function $g : V \to \mathbb R^{d-1}$ such that $g(0)=0$, and $\partial_t \nabla_y \Psi(g(t),t) = 0$ for all $t \in V$.
For a fixed $t_k \in V\cap \supp a$, let us set $y_k = g(t_k) \in U \cap {\supp {a}}$.

By the Taylor expansion of $\Psi$ at $y_k$ and then at $t_k$, we have 
\begin{align*}
\Psi(y,t) 
 =&  
\Psi(y_k,t) + \mathlarger\langle \nabla_y \Psi(y_k,t_k), y-y_k \mathlarger\rangle  + \mathlarger\langle \partial_t^2 \nabla_y \Psi(y_k,t_k) \frac{(t-t_k)^2}{2!}, y-y_k \mathlarger\rangle + \\  
 & \dots + \mathlarger\langle \partial_t^d\nabla_y \Psi(y_k,t_k) \frac{(t-t_k)^d}{d!}, y-y_k \mathlarger\rangle  + O(|y - y_k|^2 + |y-y_k||t-t_k|^{d+1}),
\end{align*}
where the first order term vanishes because of $\partial_t\nabla_y\Psi(y_k,t_k) =0$.
Let us set 
\[ 
\gamma_\circ(t) = (t^2/2!,\dots,t^d/d!).
\]
Discarding harmless factors  $\Psi(y_k, t)$ and
 $\langle\nabla_{y}\Psi(y_k,t_k),  y-y_k\rangle$,  
we  may assume that
\begin{equation} 
\label{taylor}
 \Psi(y,t)  
   =  \langle \mathcal M(t_k) \gamma_\circ(t-t_k),  y - y_k  \rangle  +O(|y - y_k|^2+ |y- y_k||t-t_k|^{d+1}).
 \end{equation}    
Here $\mathcal M(t_k)$ is the matrix of which $j$--th column vector is given by $\partial_t^{j+1} \nabla_y \Psi(y_k,t_k)$, $1\le j \le d-1$. By the assumption \eqref{M'}, $\mathcal M(t_k)$ is nonsingular on the support of $a$.

Let us fix $\delta>0$ such that  $[0,\delta] \subset V \cap \supp a$, and  take $\lambda>0$ such that  $\lambda^{-1/(2d)} <\delta$ and $\delta \lambda^{1/(2d)} =: \ell \in \mathbb N$.
%Choosing  an integer $\ell$ which satisfies $\delta/\ell \sim \lambda^{-1/(2d)}$, 
We decompose the interval $[0,\delta]$ into intervals $I_k=[t_{k-1}, t_{k}] $, $1 \le k \le \ell$, of length $|I_k|\sim \lambda^{-1/(2d)}$ such that $[0,\delta]=\bigcup_{1\le k\le \ell}  I_k $ .
On each interval $I_k$, we  observe the following.

\begin{lem}\label{const}
Let $\rho =1/(2d)$. 
Consider a rectangle $\mathcal R\subset \mathbb R^{d-1}$
defined by \[\mathcal R=\{ (x_2,\dots,x_d) : |x_j | \le c \lambda^{-1+j\rho}, \, 2 \le j \le d \, \}\]  with a small constant $c>0$.
For each interval $I_k$, let $\mathcal P_k$ be the parallelepiped defined by
\[
\mathcal P_k = \{ y: \mathcal M^T(t_k) (y-y_k) \in \mathcal R\},
\]
where $y_k = g(t_k)$ and $\mathcal M^T(t_k)$ is the transpose of the matrix of $\mathcal M(t_k)$.
If $c$ is sufficiently small, then $| \Psi(y,t) | \le \lambda^{-1}$ for $y \in \mathcal P_k$ and $t \in I_k$.
\end{lem}

\begin{proof}
Since $|t-t_k|\lesssim \lambda^{-\rho}$, we have,  for $y \in \mathcal P_k$, 
\begin{equation}
\label{haha}
 | \langle \mathcal M(t_k) \gamma_\circ(t-t_k),  y - y_k \rangle | = | \langle \gamma_\circ(t-t_k),  \mathcal M^T (t_k) ( y - y_k) \rangle  | \lesssim (d-1) c \lambda^{-1}.
 \end{equation}
If we set $\| v(t_k)\| = \max_i |v_i(t_k)|$ for the column vectors $v_i(t_k)$ of $\mathcal M^{-T}(t_k)$, then $
| y -y_k| \le  (d-1) c \|v(t_k)\|  \lambda^{-1+d\rho}$ for $y \in \mathcal P_k$.
Hence,  we obtain 
\[
| y - y_k|^2 \lesssim c^2 \lambda^{-2+2d\rho} = c^2\lambda^{-1}
\quad \mbox{and} \quad 
| y - y_k| |t-t_k|^{d+1} \lesssim c  \lambda^{-1-\rho} \ll c\lambda^{-1}.
\]
Thus, by  \eqref{taylor},  \eqref{haha}, and the above we see  $| \Psi(y,t) | \le \lambda^{-1}$ for a sufficiently small $c>0$. 
\end{proof}

To prove Proposition \ref{variable2} we need to show that the estimate 
\begin{equation}\label{Spq}
\| \mathfrak T_\lambda f 
\|_{L^q(\mathbb R^{d-1})} \le C \lambda^{-\frac{d-1}q} \| f\|_{L^{p}(I)}
\end{equation}
implies 
\begin{align}
\label{qnn}     & q\ge  (d^2+d)/2,\\ 
\label{pqnn}  1/p&+(d^2+d-2)/2q \le 1.
\end{align}

\begin{proof}[{Proof of \eqref{Spq} $\Rightarrow$ \eqref{qnn}}] 
Let $\{\epsilon_k\}_{k=0}^\ell$ be independent random variables having the values $\pm1$ with equal probability.
We set  \[f(t)= \sum_{k=0}^\ell \epsilon_k \chi_{I_k}(t)\]
and  consider the expectation $\mathbb E ( \| \sum_k \epsilon_k \mathfrak T_\lambda \chi_{I_k}\|_{L^q}^q)$. 
By Fubini's theorem and Khintchine's inequality, we get, for $1<q<\infty$,
\begin{align}\label{expectation}
	\mathbb E ( \| \sum_k \epsilon_k \mathfrak T_\lambda \chi_{I_k}\|_{L^q}^q) 
  = 
	\int \mathbb E ( |\sum_k \epsilon_k \mathfrak T_\lambda \chi_{I_k}(y) |^q ) dy 
 \sim 
	\int \Big( \sum_k |\mathfrak T_\lambda \chi_{I_k}(y) |^2 \Big)^{\frac q2} dy.
\end{align}
By Lemma \ref{const} we have $ |\lambda \Psi(y,t)|\le 1$ for $y \in \mathcal P_k$ and $t \in I_k$.  
It is easy to see  
\[
|\mathfrak T_\lambda \chi_{I_k}|^2 \gtrsim |I_k|^2\chi_{\mathcal P_k} \sim \lambda^{-1/d} \chi_{\mathcal P_k}.
\]
Thus, it follows that
\begin{align*}
	\int \Big( \sum_k |\mathfrak T_\lambda \chi_{I_k}(y) |^2 \Big)^{\frac q2} dy
\gtrsim 
	\lambda^{-\frac q {2d}} \int |\sum_k \chi_{\mathcal P_k} |^{\frac q2}\,dy
\gtrsim 
	\lambda^{-\frac q {2d}} \int \sum_k \chi_{\mathcal P_k} \,dy
=
	\lambda^{-\frac q {2d}} \sum_k |\mathcal P_k|.
\end{align*}
For the second inequality, we use the fact that $q \ge 2$.
Combining this with \eqref{expectation} and using \eqref{Spq},
we see that
\begin{align}\label{sum}
	\lambda^{-\frac{q}{2d}}\sum_{k=0}^\ell |\mathcal P_k| 
\lesssim  
	\Big\| \sum_{k=0}^\ell |\mathfrak T_\lambda \chi_{I_k} |^2 \Big\|_{L^{\frac q2}}^{\frac q2}
\sim 
	\mathbb E ( \| \sum_k \epsilon_k \mathfrak T_\lambda \chi_{I_k}\|_{L^q}^q)  
\lesssim 
	\lambda^{-(d-1)} \delta^{\frac q p}. 
\end{align}
From the definition of $\mathcal P_k$ in Lemma \ref{const}, it follows that $|\mathcal P_k|\sim \lambda^{-(d-1)+(\frac{d^2+d}{2}-1)\cdot\frac{1}{2d}}$. Since $\ell \sim \delta \lambda^{\frac{1}{2d}}$, we have
\[
\delta\lambda^{-\frac{1}{2d} \big( q -\frac{d^2+d}{2}\big)} \lesssim \delta^{\frac q p}.
\]
For a fixed constant $\delta>0$, we see that \eqref{qnn}
is necessary by letting $\lambda \rightarrow \infty$.
\end{proof}

\begin{proof}[{Proof of \eqref{Spq} $\Rightarrow$ \eqref{pqnn}}] 
Let $J \subset [0,\delta]$ be an interval of length $|J| =\lambda^{-1/(2d)}$. 
By Lemma \ref{const}, we can find a parallelepiped $\mathcal P$ such that $| \psi(y,t)| \le \lambda^{-1}$ for $y \in \mathcal P$, $t \in J$.
If we set $f=\chi_{J}$, it follows that 
\[
	\| \mathfrak T_\lambda f  \|_{L^q(\mathcal P)} 
\ge 
	C\lambda^{-\frac{1}{2d}}|\mathcal P|^{1/q}
\ge 
	\lambda^{-\frac{1}{2d}} \big( \lambda^{-(d-1)+\frac{d^2 +d -2}{2}\cdot\frac{1}{2d}} \big)^{\frac 1q} .
\]
By \eqref{Spq}, we obtain
$
\lambda^{-\frac{1}{2d}}  \lambda^{-\frac{d-1}{q}+\frac{d^2 +d -2}{2}
\cdot\frac{1}{2dq}}  \lesssim  \lambda^{-\frac{ d-1}{q}}\lambda^{-\frac{1}{2dp}}. 
$
Thus we get  \eqref{pqnn}  by letting  $\lambda \rightarrow \infty$.
\end{proof}

\begin{proof}[Proof of Theorem \ref{necessary2}]

The sufficiency part follows from Theorem \ref{HL}. 
To prove the necessity part it is enough to show that  the estimate for $T_\lambda^\gamma f$ over $\mathbb S^{d-1}$ can be reformulated to an estimate for 
$\mathfrak T_\lambda$ (see \eqref{Tf'}) while the  phase function $\Psi$ satisfies  the hypotheses \eqref{M'} and \eqref{Gcurvature} in Proposition \ref{variable2}.
 
For a given $\gamma$, we write $\gamma(t)=(\gamma_1(t),\gamma_\ast(t))
\in \mathbb R \times \mathbb R^{d-1}$. Since $\gamma$ satisfies \eqref{nonv}, we have $\gamma'(0) \neq 0$. 
By rotation 
we may assume that $\gamma$ satisfies \eqref{nonv},
\begin{equation}
\label{nonzero}
\gamma_1'(0) \neq 0, \text{ and } \gamma_\ast'(0) = 0.
\end{equation}
Then we consider the part of  $\mathbb S^{d-1}$ near $-e_1$.  That is to say, $\mathbb S^{d-1}\cap B(-e_1, \epsilon_0)$ for some small $\epsilon_0>0$.  Then,  we can parametrize $\mathbb S^{d-1}\cap B(-e_1, \epsilon_0)$ with a smooth function $\phi$ such that   
$ y \mapsto  (\phi( y)-1,  y)$ for $ y= (y_2,\dots,y_d) \in \mathbb R^{d-1}$ near the origin,
$\phi(0)=0$, $\nabla_y \phi(0)=0$, and 
\begin{equation}
\label{hessian}
\det H\phi  = \det \left(  \frac{\partial^2 \phi}{\partial y_i \partial y_j}  \right)_{2\le i, j\le d} \neq 0
\end{equation}
near $0$.  
Here $H$ denotes the Hessian matrix. 
Then, discarding the harmless constant $-1$,  it suffices to consider an oscillatory integral operator
\begin{equation}\label{Sf}
\mathfrak T_\lambda f(y) = \int_I e^{i \lambda \psi( y,t) } a(y,t)
f(t) \,dt,
\end{equation}
where $ \psi(y,t) = (\phi(y), y) \cdot \gamma(t)$ for 
$(y,t)\in\mathbb R^{d-1}\times \mathbb R$  and $a $ is a smooth cutoff function which is supported in a small enough neighborhood of the origin.  
Thus it remains to check $\psi$ satisfies \eqref{M'} and \eqref{Gcurvature} near the origin.

Since $\partial_t \nabla_y \psi(0,0) = \nabla_y\phi(0) \gamma_1'(0) + \gamma_\ast'(0) = 0$,
it remains to check that $\psi$ satisfies \eqref{M'}  and \eqref{Gcurvature} on the support of $a$. 
Because $\det\nabla_y \partial_t \nabla_y \psi(0,0) = \gamma_1'(0)\det H \phi(0) \neq 0$,
\eqref{Gcurvature} follows by continuity provided that the support of $a$ is small enough.
It remains to check that $\psi$ satisfies \eqref{M'} on the support of $a$. 
By the implicit function theorem, there exist neighborhoods $U \subset \mathbb R^{d-1}$ and $V \subset \mathbb R$ of $(0,0)$, and $g\in C^1(V)$ 
such that $g(0) =0$, $g(V) \subset U$, and 
\begin{equation}\label{ak}
\partial_t \nabla_y \psi(g(t),t) = 0 \text{ for all } t \in V.
\end{equation}
As observed in the proof of Proposition \ref{variable2}, it is enough to show that
$\psi$ satisfies \eqref{M'}  for $y =g(t)$. 
By \eqref{ak}, we have
$\partial_t\nabla_y \psi(g(t),t) = \gamma_1'(t) \nabla_y\phi(g(t)) + \gamma'_\ast(t)=0$ and
 $\gamma_1'(t) \neq 0$.
 For \eqref{M'}, we observe that 
\[
	\partial_t^{j+1} \nabla_y \psi(g(t),t) 
= 
	\partial_t^{j+1}\Big(\gamma_1(t) \nabla_y \phi(y) +  \gamma_\ast(t)\Big)\Big|_{y=g(t)}
= 
	- \frac{\gamma_1^{(j+1)}(t) }{\gamma_1'(t)} \gamma'_\ast(t) + \gamma^{(j+1)}_\ast(t).
\]
Using this we have 
\begin{align*}
&\quad\frac{1}{\gamma_1'(t)} \det(\gamma'(t),\dots,\gamma^{(d)}(t)) \\
& = 
	\det \begin{pmatrix} 1 & \gamma_1''(t) /\gamma_1'(t)&\cdots&\gamma_1^{(d)}(t) / \gamma_1'(t) \\
	\gamma_\ast'(t)  &\gamma_\ast''(t)& \cdots&
	\gamma_\ast^{(d)}(t) \end{pmatrix} \\
&  = 
	\det \begin{pmatrix} 1 &  0 &\cdots& 0  \\
	\gamma_\ast'(t)  &\gamma_\ast''(t) -\frac{\gamma_1''(t)}{\gamma_1'(t)}\gamma_\ast'(t) & \cdots&
	\gamma_\ast^{(d)}(t) -\frac{\gamma_1^{(d)}(t)}{\gamma_1'(t)}\gamma_\ast'(t)
	\end{pmatrix}\\
&  = 
	\det \mathcal (\partial_t^2 \nabla_y \psi(g(t),t),\dots, \partial_t^d \nabla_y \psi(g(t),t) ).
\end{align*}
Therefore \eqref{M'} holds since $\gamma$ is nondegenerate on $I$.
\end{proof}

\newcommand{\ceil}[1]{\left\lceil #1 \right\rceil}

\section{Proof of Theorem \ref{rmkk2} and Proposition \ref{hyper}}

We first prove   Theorem \ref{rmkk2}.

If $\gamma$ is a finite type curve, after finite decomposition, translation (also subtracting  a harmless constant) and rescaling,  we may regard the curve as the one given by a small perturbation of a monomial curve. 
Thus we are naturally led to consider the class of curve $\mathfrak G^{\mathbf a}(\epsilon)$
which is defined as follows: For $\epsilon>0$ and $\mathbf a\in \mathcal A$,
\begin{equation}
\label{finiteclass}
\mathfrak G^{\mathbf a}(\epsilon) 
= 
	\{  \gamma \in  C^{\infty}(I) : \gamma(t) = (t^{\mathbf a_1}\varphi_1(t),\dots, t^{\mathbf a_d}\varphi_d(t)),\ \|\varphi_i - 1/(\mathbf a_i!) \|_{C^{\mathbf a_d+1}(I)}  \le \epsilon \} .
	\end{equation}
In order prove Theorem \ref{rmkk2} it is enough to show the desired estimate with $\gamma\in \mathfrak G^{\mathbf a}(\epsilon)$ while the surface measure is replaced with the $(d-1)$--dimensional measure (see  Definition \ref{dimensional}). 

This type of reduction from finite type to almost monomial type already appeared in \cite[Section 3]{HamLee}, so we shall  be brief. We set $[a,b]^* = [a,b]$ if $a  < b$, or $[a,b]^* = [b,a]$ if $a >b$.  Suppose  $\gamma$ is of type $\mathbf a(t)$ at $t$ and  let  us set 
\[ M_{t}=[\gamma^{(\mathbf a_1(t))}(t), \dots, \gamma^{(\mathbf a_d(t))}(t)],  \  \quad  D^u_t= (u^{\mathbf a_1(t)}e_1,\dots,u^{\mathbf a_d(t)}e_d). \]
  Then, by Taylor's theorem, it is not difficult to see that there exists $\delta>0$ such that, if  $[t_0, t_0+u]^*\subset I$ and $|u|<\delta$, 
\begin{equation} 
	\label{exp}\gamma( u t+ t_0) - \gamma(t_0) 
= 
	M_{t_0} D^{u}_{t_0} (t^{\mathbf a_1(t_0)}\varphi_1(ut),\dots, t^{\mathbf a_d(t_0)}\varphi_d(ut)),   \quad t\in I,
\end{equation} 
where  $\varphi_i$ are smooth functions satisfying $\varphi_i(ut) = 1/(\mathbf a_j(t_0)!)  + O(\delta)$.  
Thus, for any $\epsilon>0$, there exists $\delta=\delta(\epsilon, t_0)$ such that
\[
	\gamma_{t_0}^{u}(t)
:=
	( M_{t_0} D_{t_0}^{u} )^{-1} (  \gamma( u t+ t_0) - \gamma(t_0))   \in \mathfrak G^{\mathbf a(t_0)}(\epsilon)
\]
whenever $|u| <\delta$ and $[t_0, t_0+u]^*\subset I$.  
See \cite[Lemma 3.1, Lemma 3.3]{HamLee} for details. 
 Suppose now that  $\epsilon>0$ be fixed. 
Since $I$ is compact, we can decompose $I$ into finitely many intervals $I_\ell =[t_\ell, t_\ell + u_\ell]^*$ such that $\gamma_{t_\ell}^{u_\ell}(t) \in \mathfrak G^{\mathbf a(t_\ell)}(\epsilon)$.  
Recalling   $d\sigma$ denotes the surface measure on $\mathbb S^{d-1}$, we  define a positive  measure  $d\sigma_\ell$ defined by 
\[
\int F(x) d\sigma_\ell(x)
        := 
               \int F((M_{t_\ell} D_{t_\ell}^{u_\ell})^T x ) d\sigma(x),\footnote{The definition can be justified via the Riesz representation theorem. See \cite[pp. 257--258]{HamLee}.} \ F \in C_c(\mathbb R^{d}),
\]
which is clearly a  $( d-1 )$--dimensional  measure.  By making change of variables, we see that  
 \begin{align*}
& \| T_\lambda^\gamma f \|_{L^q (\mathbb S^{d-1}) }  
   \le         
      \sum_\ell \Big\| \int_{[t_\ell,t_\ell+u_\ell]^*} e^{i\lambda x\cdot  \gamma(t)} f(t)\,dt \Big\|_{L^q(\mathbb S^{d-1})} 
           \\
              =  
                 \sum_\ell  &\Big\| \int_{I} e^{i\lambda (M_{t_\ell} D_{t_\ell}^{u_\ell})^T x \,\cdot\,{\gamma}_{t_\ell}^{u_\ell}(t)}  f_{u_\ell}(t)\,dt \Big\|_{L^q(\mathbb S^{d-1})} 
                  = \sum_\ell     \Big\| T_\lambda^{\gamma_{t_\ell}^{u_\ell}} f_{u_\ell} \Big\|_{L^q(d\sigma_\ell)},
 \end{align*}
 where $  \mathbf a(t_\ell)=(\mathbf a_1(t_\ell),\dots, \mathbf a_d(t_\ell))$, $f_{u_\ell} (t) = u_\ell f(u_\ell t+t_\ell)$.  Since there are only finitely many $\ell$, so the proof of Theorem \ref{rmkk2} reduces  to showing that,  for each $\ell$,  
 \begin{equation}\label{final}
 \Big\| T_\lambda^{\gamma_{t_\ell}^{u_\ell}} g \Big\|_{L^q(d\sigma_\ell)} \le \lambda^{-(d-1)/q} \|g\|_{L^{p,q}(I)}
\end{equation}
holds whenever $q >d(d+1)/2$ and $1/p +  \max_{t \in I} \{  \|\mathbf a(t)\|_1-\mathbf a_1(t)\} /q \le 1$.
 For this purpose, we actually prove more than what we need by replacing the $(d-1)$--dimensional measure $\sigma_\ell$ by a general $\alpha$--dimensional measure.
We basically  follow the argument in \cite{HamLee}. 

\begin{defn}\label{dimensional}
Let $\alpha\in (0, d]$ and by $B(x,r)$ we denote the ball centered at $x$ of radius $r$. Suppose that $\mu$ is a positive Borel regular measure with compact support such that
\begin{equation}\label{measure} 
\mu (B(x,r)) \le C_\mu r^\alpha \,\text{ for } (x,r) \in \mathbb R^{d}\times \mathbb R_+ 
\end{equation}
with $C_\mu>0$ independent of $x,r$. Then we say $\mu$ is $\alpha$--dimensional.
\end{defn}

For $\nu\in \mathbb R$ we denote by  $\ceil\nu$ the smallest integer which is not less than $\nu$.  For $(\mathbf a, \alpha)\in \mathcal A\times (0,d]$ we set
\[
\kappa(\mathbf a,\alpha) :=
(\alpha +1 - \ceil\alpha)\mathbf a_{d-\ceil\alpha+1}+\sum_{i=d-\ceil \alpha +2}^d \mathbf a_i, \quad   \beta(\alpha) := \kappa((1,2,\dots, d),\alpha).
\]
Thus $\kappa(\mathbf a,\alpha)\ge \beta(\alpha)$ and $\kappa(\mathbf a,\alpha)= \beta(\alpha)$ if and only if $\mathbf a=(1,2,\dots, d)$.
We also note that   $\kappa(\mathbf a,\alpha)> \beta(\alpha)$ implies  $\|\mathbf a\|_1 
 > {d(d+1)}/2$. 

\begin{prop}\label{AA}
Let $d\ge 3$. Let  $\gamma \in \mathfrak G^{\mathbf a} (\epsilon)$ for some $\mathbf a\in \mathcal A$ 
with $\|\mathbf a\|_1   - {d(d+1)}/2 \ge 1$.
Suppose that $\mu$ is a compactly supported positive Borel measure satisfying \eqref{measure} with $\alpha \in [d-1,d]$.
If $\epsilon>0$ is sufficiently small, then 
\begin{equation}\label{finitetype}
\|T_\lambda^\gamma f \|_{L^{q}(d\mu)} 
\lesssim \lambda ^{-\alpha/q} \| f \|_{L^{p,q}(I)}
\end{equation}
holds for 
$ 1/p + \kappa(\mathbf a, \alpha)/q \le 1$ 
and $q> \beta(\alpha)+1$. 
Moreover,  if $1/p + \kappa(\mathbf a, \alpha)/q > 1$,  there is a measure $\mu$ which satisfies  \eqref{measure} but the estimate \eqref{finitetype} fails.

\end{prop}

In fact, 
Proposition \ref{AA} continues to hold for $\alpha\in (0,d-1)$ under some additional conditions on $p,q$ as explained after the statement of Theorem \ref{hamlee11}.

\newcommand{\ca}{\ceil\alpha}

\begin{proof}
We decompose
$T_\lambda^\gamma f =\sum_{\ell =0}^\infty T_\ell f$
where $T_\ell f$ is defined by 
\begin{align}\label{TKF}
T_\ell f(x) &= \int_{[2^{-\ell-1},2^{-\ell}]} e^{i\lambda x \cdot \gamma(t)} f(t)\,dt.
\end{align}

For $h>0$, let us define  $\mathfrak D_h^\mathbf a=(h^{\mathbf a_1}e_1,\dots,h^{\mathbf a_d}e_d)$. For each fixed $\ell$, we define a positive Borel measure $\mu_\ell$  by setting 
\[
\int F(x) \,d\mu_\ell(x) =
2^{-\ell\,\kappa(\mathbf a, \alpha) } 
\int F( \mathfrak D_{2^{-\ell}}^{\mathbf a} x )\, d\mu(x), \  \  F\in C_c(\mathbb R^d).
\]
We now show that $\mu_\ell$  satisfies \eqref{measure}. 
Note that the set $\mathcal R = \{ y: \mathfrak D_{2^{-\ell}}^{\mathbf a} y\in B(x,r) \}$, which is 
contained in a rectangle of dimensions $C2^{\ell {\mathbf a_1}}r\times C2^{\ell \mathbf a_2} r\times \dots\times C2^{\ell \mathbf a_d}r$,
can be covered by as many as $ O( \prod^{d}_{i=d+2-\ca} 2^{\ell(\mathbf a_{i}-\mathbf a_{d+1-\ca})})
$ 
 cubes of side length $2^{\ell {\mathbf a}_{d+1-\ceil\alpha}}r$.
Thus, applying \eqref{measure} to each of these cubes,  we see that 
\begin{align*}
& \mu_\ell (B(x,r)) \lesssim 2^{ -\ell\kappa(\mathbf a, \alpha)}
\int \chi_{B(x,r)}(\mathfrak D_{2^{-\ell}}^\mathbf a y) d\mu(y) \lesssim  2^{ -\ell\kappa(\mathbf a, \alpha)}\mu(\mathcal R)
\\
&\lesssim  2^{ -\ell\kappa(\mathbf a, \alpha)} 2^{\ell((1-\ca)\mathbf a_{d+1-\ca}
+\sum_{i=d+2-\ca}^d \mathbf a_i)}
(2^{\ell\mathbf  a_{d+1-\ca}}r)^{\alpha}
\lesssim r^{\alpha}.
\end{align*}
Therefore $\mu_\ell$ satisfies \eqref{measure}.
Also we consider  $\gamma_\ell$ and $f_\ell$ which are defined   by 
\[\gamma_\ell(t):=\mathfrak D_{2^{\ell}}^{\mathbf a} \gamma(2^{-\ell}t)=(t^{\mathbf a_1}\varphi_1(2^{-\ell}t),\dots, t^{\mathbf a_d}\varphi_d(2^{-\ell}t)),   \ \  f_\ell(t) =2^{-\ell}f(2^{-\ell}t),\]
respectively.
Then, by scaling $t \rightarrow 2^{-\ell}t$ we have that  
\begin{align} \nonumber
\|T_\ell f \|_{L^{q}(d\mu)}^q
& = \int \Big|\int_{[1/2,1]} e^{i\lambda \mathfrak D_{2^{-\ell}}^{\mathbf a} x \cdot \gamma_\ell(t) } f_\ell(t) \,dt \Big|^q \,d\mu(x) \\
&=2^{\ell\kappa(\mathbf a, \alpha)}
\int \Big|\int_{[1/2,1]} e^{i\lambda x \cdot 
\gamma_\ell (t)}f_\ell (t)\,dt \Big|^q \,d\mu_\ell (x) \label{tkfq}.
\end{align}
 We now use the following to get a bound for each  $T_\ell$, which is a special case of Theorem 1.1 in  \cite{HamLee}.

\begin{thm}[Theorem 1.1 in \cite{HamLee}]
\label{hamlee11}
Let $d \ge 3$ and  $\alpha\in [d-1,d]$. 
Suppose that $\gamma$ satisfies \eqref{nonv} and $\mu$ is $\alpha$--dimensional. 
Then, for $p,q$ satisfying 
\begin{equation}
\label{qpalpha}
	1/p+\beta(\alpha)/q<1, \quad q >\beta(\alpha)+1,
\end{equation}
we have the estimate 
$	\|T_\lambda^\gamma f\|_{L^{q}(d \mu)}
\lesssim 
	\lambda^{-\alpha/q}\|f\|_{L^{p}(I)}$. 
\end{thm}

Actually the estimate is valid on a wider range of $\alpha,$ $p,$ and $q$ but  this is not relevant to our purpose.
The additional restriction 
$ d/q \le (1-1/p)$ and $q\ge  2d$ which is in Theorem 1.1 in  \cite{HamLee} is not necessary here because $d \ge 3$ and  $\alpha\in [d-1,d]$.

The bound $\|T_\lambda^\gamma\|_{p\to q}$ in Theorem \ref{hamlee11} is stable under small smooth perturbation of $\gamma$.  
Since $\gamma_\ell \in \mathfrak G^\mathbf a (\epsilon)$
(in fact, $\gamma_\ell (t)\in  \mathfrak G^{\mathbf a}(C 2^{-\ell}\epsilon)$ for some constant $C>0$),
the torsion $\tau_\ell$ of $\gamma_\ell$ at $t$ is $|\tau_\ell (t)| \sim   t^{\|\mathbf a\|_1 - d(d+1)/2} $
where the implicit constant is independent of $\ell$
(see Lemma 3.4 in \cite{HamLee}).
Thus, choosing a sufficiently small $\epsilon>0$, we see that  $\gamma_\ell$ is a small smooth perturbation of the curve  $(\frac{t^{\mathbf a_1}}{\mathbf a_1!}, \dots, \frac{t^{\mathbf a_d}}{\mathbf a_d!})$,  which is nondegenerate on  the interval $[1/2,1]$.
Recalling that $\mu_\ell $ satisfies \eqref{measure} with $\mu=\mu_\ell$, 
we apply Theorem \ref{hamlee11} to \eqref{tkfq} and obtain, for $p,q$ satisfying \eqref{qpalpha}, 
\begin{equation}\label{tkfmu}
\| T_\ell f\|_{L^q(d\mu)}^q
\lesssim 2^{\ell q \big(\frac 1p +\frac{\kappa(\mathbf a, \alpha)}q -1\big)} \lambda^{-\alpha} \|f\|_p^q.
\end{equation}

\renewcommand{\B}{\Big}

Using  this, we can get a weak type estimate for $T_\lambda^\gamma$ on the critical line $1/p+\kappa(\mathbf a, \alpha)/q =1$.  
With an integer  $N$ which is to be chosen later, we consider
\[\mu( \{x: | T_\lambda^\gamma f(x)| >\delta \} )
\le \mu\B( \B\{x: |\sum_{\ell=-\infty}^N T_\ell f(x)| >\frac\delta2 \B\} \B)
+\mu \B(\B\{x: |\sum_{\ell=N+1}^\infty T_\ell f(x)| >\frac\delta2 \B\} \B).
\]
Here we trivially extend $T_\ell$ to $\ell =-1,-2,\dots$ by setting $T_\ell =0$.
{Now, fixing $p,q$ satisfying $ 1/p + \kappa(\mathbf a, \alpha)/q =1$ 
and $q> \beta(\alpha)+1$, we show the estimate \eqref{finitetype}.} We choose $1 \le  q_1, q_2 \le \infty$ such that \eqref{qpalpha} holds {with $(p,q)=(p, q_i)$, $i=1,2$, }
$1/p+\kappa(\mathbf a, \alpha)/q_1 >1 $, 
and $1/p+\kappa(\mathbf a, \alpha)/q_2 <1 $.  Such choices are possible since $\kappa(\mathbf a,\alpha)> \beta(\alpha)$. 
Since $1/p+\kappa(\mathbf a, \alpha)/q_2 -1<0<1/p+\kappa(\mathbf a, \alpha)/q_1 -1 $, 
by Chebyshev's inequality and Minkowski's inequality, and then  making use of \eqref{tkfmu}, we have
\begin{align*}
&\mu( \{x: | T_\lambda^\gamma f(x)| >\delta \} )\lesssim \delta^{-q_1} \big( \sum_{\ell=-\infty}^N \| T_\ell f \|_{L^{q_1}(d\mu)} \big)^{q_1}
+\delta^{-q_2} \big( \sum_{\ell=N+1}^\infty \| T_\ell f \|_{L^{q_2}(d\mu)} \big)^{q_2} 
\\
\lesssim &\delta^{-q_1} \big( 2^{N(1/p+\kappa(\mathbf a, \alpha)/q_1 -1)}
\lambda^{-\alpha/q_1} \| f\|_{p}  \big)^{q_1}
+
\delta^{-q_2}  \big( 2^{N(1/p+\kappa(\mathbf a, \alpha)/q_2 -1)}
\lambda^{-\alpha/q_2} \| f\|_{p}  \big)^{q_2}.
\end{align*}
Taking $N$ such that $2^N \sim \delta^{-p'} \|f\|_{p}^{p'}$, we get 
$ \mu( \{x: | T_\lambda^\gamma f(x)| >\delta \} )\lesssim  \delta^{-q}\lambda^{-\alpha}\|f\|_p^q$ for  $p,q$  satisfying 
$1/p+\kappa(\mathbf a, \alpha)/q =1$,  and hence $T_\lambda^\gamma$ is of weak type $(p,q)$.
By real interpolation along the resulting estimates  and H\"{o}lder's inequality, we get
\eqref{finitetype} for $1/p+\kappa(\mathbf a, \alpha)/q \le  1$.

Now we show that  the condition $1/p+\kappa(\mathbf a, \alpha)/q \le 1$
is necessary
for \eqref{finitetype}.
Let us consider the measure $d\mu$ which is defined by 
\[
d\mu(x) =\chi_{B(0,1)}\prod_{i=1}^{d-\ceil \alpha} d\delta(x_i) |x_{d-\ca+1}|^{\alpha-\ceil \alpha}
dx_{d-\ca+1}dx_{d-\ca+2}\dots dx_d
\]
Here $d\delta$ is the one dimensional Dirac measure. 
It is easy to check that $\mu$ satisfies \eqref{measure}.
If we take $f(t) = \chi_{[0,\lambda^{-\rho}]}(t)$ for some $\rho>0$,
then $|T_\lambda^\gamma f(x)| \ge C \lambda^{-\rho}$
whenever $x \in \mathcal R_{\mathbf a}=\{ x\in \mathbb R^d :
|x_i|\le c \lambda^{-1+\rho \mathbf a_i} \}$ for a small $c>0$.
Since $\|f\|_{p,q} \sim \|f\|_p=\lambda^{-\rho/p}$
and 
\[
\mu(\mathcal R_{\mathbf a}) \sim 
(\lambda^{-1 + \rho \mathbf a_{d-\ceil\alpha+1}})^{\alpha -\ceil\alpha+1} \lambda^{1-\ceil\alpha+\rho \sum_{i=d-\ceil\alpha+2}^d \mathbf a_i} = \lambda^{-\alpha + \rho\kappa(\mathbf a, \alpha)},
\]
the estimate \eqref{finitetype} implies 
$
\lambda^{-\rho} \big(
\lambda^{-\alpha} \lambda^{\rho \kappa(\mathbf a, \alpha)}\big)^{1/q}
\le C \lambda^{-\alpha/q} \lambda^{-\rho/p}.
$
Taking $\lambda$ which  tends to $\infty$ gives the desired condition   $1/p+\kappa(\mathbf a, \alpha)/q \le 1$. 
\end{proof}

\begin{proof}[Proof of \eqref{final}] To begin with  we recall that $\gamma_{t_\ell}^{u_\ell}\in \mathfrak G^{\mathbf a(t_\ell)}(\epsilon)$.
It is obvious that $\sigma_\ell$ satisfies \eqref{measure} with $\alpha=\ceil\alpha = d-1$. So, we have $\beta(d-1) = d(d+1)/2-1$ and $\kappa(\mathbf a(t_\ell),d-1) =\| \mathbf a(t_\ell)\|_1 - \mathbf a_{1}(t_\ell)$.  We consider the two cases: $( \mathbf A)$ $\|\mathbf a(t_\ell)\|_1 -d(d+1)/2\ge1$ and $(\mathbf B)$ $\|\mathbf a(t_\ell)\|_1 -d(d+1)/2<1$, separately. 

If $\|\mathbf a(t_\ell)\|_1 -d(d+1)/2\ge1$,  applying Proposition \ref{AA}, we obtain  \eqref{final}
for $q > d(d+1)/2$ and $1/p + (\|\mathbf a(t_\ell)\|_1 - \mathbf a_{1}(t_\ell))/q \le 1$.  If  $\|\mathbf a(t_\ell)\|_1 -d(d+1)/2<1$,
then $ \|\mathbf a(t_\ell)\|_1 =d(d+1)/2$, i.e., $\mathbf a(t_\ell)=(1,2,\dots, d)$.  Thus, the curve $\gamma_{t_\ell}^{u_\ell}$ is now nondegenerate, that is to say, $\gamma_{t_\ell}^{u_\ell}$ satisfies \eqref{nonv} with $\gamma={\gamma_{t_\ell}^{u_\ell}}$. 
Regarding  this case,  we may directly apply Theorem \ref{hamlee11} to get the strong type $L^p -L^q(d\sigma_\ell)$ estimate for $T_\lambda{}^{\gamma_{t_\ell}^{u_\ell}} $ provided that $q > d(d+1)/2$ and $1/p + (d(d+1)-2)/(2q )< 1$.   Now we note that $(d(d+1)-2)/2< \max_\ell \{ \|\mathbf a(t_\ell)\|_1 - \mathbf a_{1}(t_\ell)\} $ because $ \|\mathbf a(t_0)\|_1- d(d+1)/2 \ge 1$ for some $t_0\in I$.  
This is clear since $\mathbf a_i(t_\ell)\ge i$, $i=1,\dots,d$, and $\mathbf a_i(t_\ell)>i_0$ for some $1\le i_0\le d$. Therefore, combining  the estimates for the cases $(\mathbf A)$ and $(\mathbf B)$  we get \eqref{final} whenever  $q >d(d+1)/2$ and $1/p +  \max_\ell \{ \|\mathbf a(t_\ell)\|_1 - \mathbf a_{1}(t_\ell)\}/q \le 1$.  
This completes the proof. 
\end{proof}
 
This shows the sufficiency part Theorem \ref{rmkk2} and we now turn to proof of the necessity part of Theorem \ref{rmkk2}, which is slightly more involved since we need to deal with higher order derivatives.
 
\begin{proof}[Proof of the necessity part of Theorem \ref{rmkk2}]
We show the condition $1/p +  \max_{t\in I}\{ \|\mathbf a(t)\|_1 - \mathbf a_{1}(t)\}/q \le 1$  is necessary for \eqref{fnt}.  Let $t_0\in I$ be  the point where $\gamma$ is of type $\mathbf a$ 
at $t_0$ and $ \|\mathbf a\|_1 - \mathbf a_{1}=\max_{t\in I}\{ \|\mathbf a(t)\|_1 - \mathbf a_{1}(t)\}.$ 
It suffices to show that \eqref{fnt} implies $1/p+(\|\mathbf a\|_1-\mathbf a_1)/q \le 1$ provided $f$ is supported in $[t_0,t_0\pm \epsilon_0]^*\subset I$.  
We only  consider the case $[t_0,t_0+ \epsilon_0]\subset I,$ and the other case can be handled similarly. 
From Taylor's expansion we have (see \cite[Section 3]{HamLee})  that, for $t\in [0,\epsilon_0]$, 
\begin{align}
\label{eeqq}
	\gamma(t+t_0)&-\gamma(t_0)
=
	\gamma^{(\mathbf a_1)}(t_0)\frac{t^{\mathbf a_1}}{\mathbf a_1!}\big(1 +O(t)\big)+\\
&\qquad 
	\gamma^{(\mathbf a_2)}(t_0)\frac{t^{\mathbf a_2}}{\mathbf a_2!}\big(1 +O(t)\big)+ \dots+\gamma^{(\mathbf a_d)}(t_0) \frac{t^{\mathbf a_d}}{\mathbf a_d!}\big(1 +O(t)\big).\nonumber
\end{align}
Since $\gamma^{(\mathbf a_1)}(t_0),\dots, \gamma^{(\mathbf a_d)}(t_0)$ are linearly independent, we can  
choose orthonormal vectors $\mathbf v_1,\dots, \mathbf v_{d-1}$ one after another  such that, for $i=1, \dots, d-1$, 
\[              
	\mathbf v_i \perp  \spn{\gamma^{(\mathbf a_1)}(t_0), \dots,  \gamma^{(\mathbf a_{d-i})}(t_0) },  \   \mathbf v_{i+1}   \in\spn{\gamma^{(\mathbf a_1)}(t_0), \dots,  \gamma^{(\mathbf a_{d-i})}(t_0) } . 
\]  
Additionally, let $\mathbf v_d$ be the unit vector such that $\mathbf v_d\perp \spnn{\mathbf v_1,\dots, \mathbf v_{d-1}}$. 
For  $\mathbf y=(\mb y_1,\dots, \mb y_{d-1})$  we parametrize  the part of $\mathbb S^{d-1}$ near $-\mathbf v_d$
by $\mb y \in \mathbb R^{d-1} \mapsto \mb y_1\mathbf v_1+\dots +\mb y_{d-1}\mathbf v_{d-1}+ (\phi(\mb y)-1)\mathbf v_d$ such that  $\phi(0)=\nabla\phi(0) = 0$. In fact, $\phi(\mb y)=1-\sqrt{1-|{\mb y}|^2}$.  Thus, the measure on $\mathbb S^{d-1}$ is given by $d\mu = (1 + |\nabla \phi(\mb y)|^2)^{1/2} d\mb y$. 

For some small enough $\epsilon_0>0$ let us set
\[ \overline T_\lambda  f(\mb y)=  \chi_{B(0,\epsilon_0)}(\mb y)\int e^{i\lambda \Phi(\mb y, t)}  f(t) \chi_{[0, \epsilon_0]}(t)dt, \] 
where $\Phi(\mb y,t)=(\sum_{i=1}^{d-1} \mathbf y_i \mathbf v_i+(\phi(\mb y)-1)\mathbf v_d)\cdot  (\gamma(t+t_0)-\gamma(t_0))$.
Subtracting harmless factors,  it is sufficient  to consider, instead of $T_\lambda^{\gamma}$, the operator  $\overline T_\lambda$
and to show the estimate
\begin{equation}
\label{eqq}
\|\overline T_\lambda  f\|_{L^q} \le C\lambda^{-\frac{d-1}q}\|f\|_{L^{p,q}}
\end{equation}
 implies $1/p+(\|\mathbf a\|_1-\mathbf a_1)/q \le 1$.   With a small enough $c>0$ and $0<\rho<(2\mathbf a_d-\mathbf a_1)^{-1}$ let us set
 \[
\mathcal R_{\mathbf a}
=\{ \mb y \in \mathbb R^{d-1} : |\mb y_i| \le c\lambda^{-1+\rho\mathbf a_{d+1-i} }, 1\le i \le d-1 \}.\]
We now recall that $a_1<\cdots<a_d$.  
By the choice of $\mathbf v_1,\dots, \mathbf v_{d-1}$ and  using \eqref{eeqq}  and  $\phi(\mb y)=O(|\mb y|^2)$ we notice that, for $t\in [0,\lambda^{-\rho}]$ and $\mb y\in \mathcal R_{\mathbf a}$, 
\begin{align*}
\Phi(\mb y,t)
&= 
	\sum_{i=1}^{d-1} \mb  y_i \mb v_i\cdot  \Big( \sum_{j=d+1-i}^d \gamma^{(\mathbf a_j)}(t_0)\frac{t^{\mathbf a_j}}{\mathbf a_j!}\big(1 +O(t)\big)\Big)  +O(|\mb y|^2 |t|^{\mb a_1})\\
&= 
	\sum_{i=1}^{d-1} O(|\mb  y_i||t|^{\mathbf a_{d+1-i}}) + O(|\mb y|^2 |t|^{\mb a_1})=O(c\lambda^{-1}). 
\end{align*} 
Taking sufficiently small $c>0$, we have 
$|\Phi(\mb y,t)|\le 10^{-2}\lambda^{-1}$ if $\mb y\in \mathcal R_{\mathbf a}$ and $t\in [0,\lambda^{-\rho}]$. Therefore, if we take  $f = \chi_{[0,\lambda^{-\rho}]}$ with a large $\lambda$, we see that 
$|\overline  T_\lambda  f |\gtrsim  \lambda^{-\rho}$ on $\mathcal R_{\mathbf a}$.  Since $|\mathcal R_{\mathbf a}| \sim \lambda^{-(d-1)+ \rho (\|\mathbf a\|_1-\mathbf a_1)}$, the estimate \eqref{eqq} implies 
\[
\lambda^{-\rho} \lambda^{-\frac{d-1}q+(\|\mathbf a\|_1-\mathbf a_1)\frac\rho q}
\le C  \lambda^{-\frac{d-1}q} \lambda^{-\frac\rho p}.
\]
Thus, taking $\lambda\to \infty$ we see that  $1/p+(\|\mathbf a\|_1-\mathbf a_1)/q \le 1$ is necessary for \eqref{eqq}.  This completes the proof. 
\end{proof}

We prove Proposition \ref{hyper} by making use of Proposition \ref{AA}.

\begin{proof}[Proof of Proposition \ref{hyper}]
We may assume $H  = \{ x \in \mathbb R^d : x\cdot \mathbf n =0\}$ for a nonzero vector $\mathbf n=(n_1, \dots, n_d)$. 
We may assume that $n_k\neq 0$ for some $k$.\footnote{The choice of $k$ is not important since the type of curve does not change under affine transformation.}
Then $H$ is parametrized by $x_k = \mathbf h \cdot \overline x$, where each element $h_i$ of $\mathbf h$ is given by $h_i = -n_{i}/n_k$, $1 \le i \neq k \le d$ and $\overline x =  (x_1,\dots,\widetilde{x_k},\dots,x_d) \in \mathbb R^{d-1}$. 
Here $\widetilde{x_k}$  means the omission of the $k$--th element $x_k$.

Hence, we have  
\begin{equation}
\label{phase}
	x\cdot \gamma(t)
=  
	\overline x \cdot \gamma_{\mathbf h}(t), \   \  \gamma_{\mathbf h}(t) 
:= 
	\overline{\gamma(t)} +  \gamma_k(t)  {\mathbf h}.
\end{equation}

To prove the sufficiency part of Proposition \ref{hyper}, it suffices to show that
\begin{equation}\label{TH1}
 \| T_{\lambda}^{\gamma_\mathbf h} f\|_{L^q(\mathbb R^{d-1})}\lesssim \lambda^{-(d-1)/q} \|f\|_{L^{p,q}(I)}     
\end{equation}
holds 
if $q>\frac{d(d-1)}2+1$ and $ 1/ p+   \max_{t\in I} \omega( t)/q \le 1$.

For $t\in I$, we may assume that $\omega(t)$ attains its minimum at $\mathbf a(t)=
(\mathbf a_1(t), \dots, \mathbf a_{d-1}(t))$.
Then we have 
\begin{align*}
  \det \begin{bmatrix}   \gamma_\mathbf h^{(\mathbf a_1)}(t),&  
&\cdots,& \gamma_\mathbf h^{(\mathbf a_{d-1})}(t)
\end{bmatrix}
=
\frac{(-1)^{k+1}}{n_k}
 \det \begin{bmatrix} \mathbf n,& \gamma^{(\mathbf a_1)}(t),%& \gamma^{(\mathbf a_2)}(t),
&\cdots,& \gamma^{(\mathbf a_{d-1})}(t)
\end{bmatrix} 
\neq 0,
\end{align*}
where we make use of elementary row operations  since $n_k \neq 0$.
Hence $\gamma_{\mathbf h}$ is of type $\mb a(t)= (\mathbf a_1,\dots,\mathbf a_{d-1}) $.
 %It is clear that $\gamma_{\mathbf h}$ is of finite type.

As in the proof of Theorem \ref{rmkk2}, we may work with a function $f$ supported in sufficiently a small neighborhood of $t$. 
Thus, if  $\|\mathbf a(t)\|_1 > d(d-1)/2$,  we note that $\kappa(\mathbf a(t), d-1) = \| \mathbf a(t)\|_1$ and  $\beta(d-1) = d(d-1)/2$.  Hence, by Proposition \ref{AA} with $\mu$ which is the Lebesgue measure in $\mathbb R^{d-1}$,  we get \eqref{TH1}  for $q > d(d-1)/2 +1$ and
 $1/p + \|\mathbf a(t)\|_1/q \le 1$ provided $f$ is supported in a small neighborhood of $t$.  On the other hand, if $\|\mathbf a(t)\|_1 = d(d-1)/2$, the curve $\gamma_{\mb h}$ is a nondegenerate (in $\mathbb R^{d-1}$) near the point $t$.   
In this case,  the desired estimate follows by the typical Fourier restriction estimates for nondegenerate curves in $\mathbb R^{d-1}$ (see \cite{Drury, BakLee, BOS09} for example). Thus we get \eqref{TH1} for $q > d(d-1)/2 +1$ and $1/p + d(d-1)/2q \le 1$ whenever $f$ is supported near the point $t$.  Since $I$ is compact, combining those two types of local results we obtain \eqref{TH1} for 
$q > d(d-1)/2 +1$ and $1/p + \max_{t\in I}\|\mathbf a(t)\|_1 / q \le 1$.

This range is optimal because  the conditions $q > d(d-1)/2 +1$ and $1/p + \max_{t\in I}\|\mathbf a(t)\|_1 / q \le 1$ are necessary for \eqref{TH1}.  The first one is obvious because we can not have \eqref{TH1} for  $q \le  d(d-1)/2 +1$ in $\mathbb R^{d-1}$ even for the nondegenerate curve as is mentioned in the introduction.  The necessity of the second condition can be shown by following the argument in the proof of  the necessity part of Theorem \ref{rmkk2}.    
\end{proof}

\begin{rmk} The projection of a nondegenerate polynomial curve in $\mathbb R^{d}$ to  $(d-1)$--dimensional hyperplane can be seen as a degenerate polynomial curve in $\mathbb R^{d-1}$. So, Proposition \ref{hyper} also can be deduced from the
Fourier restriction theorem for polynomial curves with affine arclength measure (see \cite{Sjolin, Sogge, Christ, DM85, DM87, BOS08, Stovall}).
\end{rmk}

\section{Details on Remarks}

\subsection{Proof of Remark \ref{k-dim}}
Let $S$ be a $k$--dimensional surface in $\mathbb R^d$.
Also let $\gamma(t) = (P_1(t), \dots, P_d(t) )$ for polynomials $P_i$ of degree $i$.
Thus, $\gamma$ satisfies \eqref{nonv}.
For $l=d-k$, we parametrize $S$ by $y=(y_1,\dots,y_{k}) \mapsto (\phi_1(y),\dots,\phi_{l}(y),y)$.
We intend to find $\phi_1, \dots,\phi_l$ such that
the phase function $\psi(y,t)=(\phi_1(y),\dots,\phi_{l}(y),y)\cdot \gamma(t)$
satisfies 
\begin{equation}
\label{abcd}
\partial_t \nabla_y \psi(y,t)=\dots=\partial_t^{l} \nabla_y \psi(y,t)=0,
\end{equation}
	and 
\begin{equation}
\label{ABCD}
\det(\partial_t^{l+1} \nabla_y \psi,\dots,\partial_t^{d}\nabla_y \psi) (y, t) \neq 0
\end{equation}
when $y=g(t)$ for some $g(t)$.
	
	Let us write $\gamma=(\gamma_a,\gamma_b) \in \mathbb R^{l} \times \mathbb R^{k}$ and  we set 
	\begin{align*}
	A_1(t)=(\gamma_a',\dots,\gamma_a^{(l)})(t), \  A_2(t)=(\gamma_a^{(l+1)},\dots,\gamma_a^{(d)})(t),\\
	B_1(t)=(\gamma_b',\dots,\gamma_b^{(l)})(t), \  B_2(t)=(\gamma_b^{(l+1)},\dots,\gamma_b^{(d)})(t).
	\end{align*}
	Since 
	$\gamma$ is nondegenerate, 
	by changing coordinates  we may assume that $A_1(t)$ is invertible. 
	Now we note that  \begin{equation}
	\label{psipsi} \psi(y,t)=(\phi_1,\dots,\phi_{l})\cdot \gamma_a(t) + y\cdot \gamma_b(t)
	\end{equation}
	and
	\begin{equation}\label{AB}
	(\partial_t \nabla \psi,\dots,\partial_t^{l} \nabla \psi)(g(t),t)
	=(\nabla_y \phi_1,\dots,\nabla_y\phi_{l})(g(t))A_1(t)
	+B_1(t).
	\end{equation}
	Thus \eqref{abcd} follows if 
	\begin{equation}\label{ABC}
	(\nabla_y \phi_1,\dots,\nabla_y\phi_{l})(g(t))=
	-B_1(t) A_1^{-1}(t)
	\end{equation}
	To obtain $\phi_1, \dots,\phi_l$ satisfying 
	 \eqref{ABC} for some $g$, we simply take
	$g(t)=(t,\dots,t)$ and   set
	\[  \phi_j(y_1,\dots,y_k)=\sum_{i=1}^k \int_{0}^{y_i} \big[-B_1(t) A_1^{-1}(t)\big]_{ij} dt,
	\] 
where $[M]_{ij} $ denotes the $(i,j)$--th element of the matrix $M$.
	 Then  \eqref{abcd} clearly holds.  
%Let us set  

Now we  show that  \eqref{ABCD} holds with our choices of $\phi_1, \dots,\phi_l$ and $g$.
From \eqref{psipsi} it follows that 
	$(\partial_t^{l+1} \nabla_y \psi,\dots,\partial_t^{d}\nabla_y \psi)(g(t),t)
	=(\nabla_y \phi_1,\dots, \nabla_y \phi_{l})(g(t))A_2(t)+B_2(t)$. Hence, using 
	\eqref{ABC}, we see that 
	\[ (\partial_t^{l+1} \nabla_y \psi,\dots,\partial_t^{d}\nabla_y \psi)(g(t),t) = B_2(t)-B_1(t) A_1^{-1}(t)A_2(t).\] 
	We recall the identity concerning the determinant of block matrix 
	\begin{align*}
	\det \begin{pmatrix}
	B_2(t) & B_1(t) \\
	A_2(t) & A_1(t)
	\end{pmatrix}
	=\det\Big(B_2(t)-B_1(t) A_1^{-1}(t)A_2(t)\Big) \det A_1(t).
	\end{align*}
	Since $\gamma$ is nondegenerate,  the determinant in the left-hand side is nonzero. Recall $A_1(t)$ is invertible and therefore \eqref{ABCD} holds. 
	
Once we have \eqref{abcd} and \eqref{ABCD} for some $g$, we can repeat the same argument as in the proof of Proposition \ref{variable2}. 
In fact, as before we partition  $I=[0,1]$ such that $I= \cup_m I_m $ and  $I_m=[t_m,t_{m+1}]$	of length $\sim \lambda^{-1/(2d)}$.
	Let  $\mathcal M(t_m)$ be the $k \times k$ matrix
	whose $j$--th column vector is $\partial_t^{d-k+j} \nabla_y \psi(g(t_m),t_m)$.
	For the rectangle $\overline{\mathcal R}$ which is given by
	\[
	\overline{\mathcal R} = \{ (x_{d-k+1},\dots,x_d) \in \mathbb R^k :
	 |x_{j} | \le c\lambda^{-1+j\rho},\,\,
	d-k+1 \le j \le d \},
	\]
	we consider the parallelepiped defined by 
	\[
	\overline{\mathcal P}_m =\{ y \in \mathbb R^k : \mathcal M^T(t_m)(y-g(t_m)) \in \overline{\mathcal R}\}.
	\]
	By the same argument as in the proof of Lemma \ref{const} to $\overline{\mathcal P}_m$
	(instead of $\mathcal P_k$), one can easily see 
	$|\psi(y,t)| \le \lambda^{-1}$ whenever $y \in \overline{\mathcal P}_m$
	and $t \in I_m$.  
Then, we repeat the argument in the proof of Proposition \ref{variable2}. 
	The only difference is that the size of ${\mathcal P_k}$
	 is now replaced by  $|\overline{\mathcal P}_m| =\lambda^{-k+
	(\frac{d^2+d}2-\frac{(d-k)^2+(d-k)}2) \frac 1 {2d}}$.
	Using this for \eqref{sum}, we see that the estimate 
	$ \|T_\lambda^\gamma f\|_{L^q(S)} \le C \lambda^{-\frac{k}{q}} \| f\|_{L^p(I)}$ implies that 
\[
\lambda^{-\frac{q}{2d}}\sum_{m}^{\sim \lambda^{1/2d}}  |\overline{\mathcal P}_m|
\lesssim \lambda^{- k} \|f\|_{L^p(I)}^q.
\]
This yields $
 \lambda^{-\frac{q}{2d}}  \lambda^{\frac{1}{2d}} 
  \lambda^{-k+\frac{1}{2d}(\frac{d^2+d}{2}-\frac{(d-k)^2+(d-k)}2)}
  \lesssim \lambda^{-k}.
$
Hence, by letting $\lambda\to \infty$  it follows that the condition $q \ge (2d-k+1)k/2+1$ is necessary.  
\qed

\subsection{Failure of $L^{p,1}(\mathbb S^{d-1})- L^{2d/(d-1),\infty}(\mathbb R^d)$ for $\widehat{fd\sigma}$}
We now shows 
the failure of $L^{p,1}(\mathbb S^{d-1})- L^{2d/(d-1),\infty}(\mathbb R^d)$ of $f \mapsto \widehat{fd\sigma}$ for any $p > 2d/(d-1)$.
This improves results in \cite{Beckner4} where the estimate  $L^{p,1}(\mathbb S^{d-1})-L^{2d/(d-1),\infty}(\mathbb R^d)$, $p= 2d/(d-1)$ fails. 	
 		 
We take  a small $\delta>0$ and decompose $\mathbb S^{d-1}$ into spherical caps $U_j$ of diameter $\delta$.
Let $\mathcal T_j$ be the tube centered at $0$ which is dual to $U_j$ with the short axes of size $c\delta^{-1}$ and the long axis of size $c\delta^{-2}$ for a sufficiently small $c>0$.
We denote by $\mathcal T_j+a_j$ the translation of $\mathcal T_j$
by $a_j \in \mathbb R^d$.
		
The following lemma is
the Kakeya set construction appeared in \cite[Lemma 3]{Beckner4}.
		
\begin{lem}
\label{log}
Let $0<\delta \ll 1$, and let $U_j$ and $\mathcal T_j$ are given as above.
Then there exists $\{a_j\}_{ 1 \le j \lesssim \delta^{-(d-1)}}$ satisfying
\[
|\bigcup_j (\mathcal T_j+a_j) |
\lesssim 
\frac{\log\log 1/\delta}{\log 1/\delta}	\sum_j | \mathcal T_j+a_j|.
\]
\end{lem}
		
		To show the failure for $p>\frac{2d}{d-1}$, it suffices to show the case $p=\infty$ and the other case follows since $L^\infty(\mathbb S^{d-1})\subset L^{p,1}(\mathbb S^{d-1})$ for any $p<\infty$.
		Let  $q_*=\frac{2d}{d-1}$ and 
let us assume that $\| \widehat{fd\sigma} \|_{L^{q_*,\infty}} \lesssim \|f\|_{L^\infty}$. We show this lead to a contradiction.

		As in the proof of Theorem \ref{necessary2},
		for each $j$, let $\epsilon_j = \pm1$ be the random variables
		with equal probability.
Let us set $ f = \sum_j \epsilon_j f_j$ where $f_j(\xi)=\chi_{U_j}(\xi)e^{-ia_j\cdot \xi}$.	
Then by Khintchine's inequality we obtain
		\begin{align*}
		\| \sum_j |\widehat{f_jd\sigma}|^2 \|_{L^{q_*/2,\infty}}^{1/2}
		= \| \big(\sum_j |\widehat{f_jd\sigma}|^2\big)^{1/2} \|_{L^{q_*,\infty}}
		\sim\| \mathbb E (  |\sum_j \epsilon_j \widehat{f_jd\sigma} |) \|_{L^{q_*,\infty}}.
		\end{align*}
		By  Minkowski's integral inequality, it follows that
\begin{align}
\label{sum2}
	\| \sum_j |\widehat{f_jd\sigma}|^2 \|_{L^{q_*/2,\infty}}^{1/2}
\lesssim
	\mathbb E ( \| \sum_j \epsilon_j \widehat{f_jd\sigma}\|_{L^{q_*,\infty}})
\lesssim 
	\|f\|_\infty.
\end{align}
For the second inequality we use the assumption $\| \widehat{fd\sigma} \|_{L^{q_*,\infty}} \lesssim \|f\|_{L^\infty}$.		% 	
		Since $\widehat{f_jd\sigma}$ is essentially constant on $
		\mathcal T_j+a_j$, we note that
		\[|\widehat{f_jd\sigma} |^2 \gtrsim |U_j|^2\chi_{\mathcal T_j+a_j}\sim \delta^{2(d-1)}
		\chi_{\mathcal T_j+a_j}.\]
       Thus, it follows that
		\begin{align*} 
		\sum_{j }  |\mathcal T_j+a_j| 
		&\le \int \sum_{j }
		\chi_{\mathcal T_j+a_j} (y) dy
		\le \Big\| \sum_{j } \chi_{\mathcal T_j+a_j} \Big\|_{L^{{q_*}/ 2,\infty}} 
		\Big| \bigcup_j \mathcal T_j+a_j \Big|^{1-2/{q_*} }\\
		&\lesssim \delta^{-2(d-1)}\Big\| \sum_{j }
		\big| \widehat{f_jd\sigma} \big|^2 \Big\|_{L^{{q_*}/2,\infty}}
		\Big| \bigcup_j \mathcal T_j+a_j \Big|^{1- 2/ {q_*} }.
		\end{align*}
		Combining this with Lemma \ref{log} and \eqref{sum2},
		we obtain
		\begin{align*}
		\Big( \sum_{j=0}^\ell | \mathcal T_j+a_j| \Big)^{ 2/{q_*}}
		\lesssim 	\delta^{-2(d-1)}\Big(\frac{\log\log 1/\delta}{\log 1/\delta} \Big)^{1- 2/{q_*}}.
		\end{align*}
		Note that  
		$
		( \sum_{j=0}^\ell | \mathcal T_j+a_j| )^{ 2/{q_*}}
		\gtrsim (\delta^{-(d-1)} \delta^{-(d+1)})^{(d-1)/d}$. 
		Since $1-2/q_* = 1/d >0$, we have a contradiction as $\delta \rightarrow 0$.
		This completes the proof.

\section*{acknowledgment} 
S. Ham was supported  by NRF-2017R1C1B2002959,  
H. Ko  was supported in part  by NRF-0450-20190054, and 
S. Lee was partially supported by NRF-2018R1A2B2006298. S. Lee would like to thank Jong-Guk Bak and Andreas Seeger for the discussion on related subjects.

%%%%%%%%%%%%%%%%%%%%%%%%%%%%%%%%%%%%%%%%%%%%%%%%%%%%%%%%%%%%%%%%%%%%%%%%
\bibliographystyle{plain}

\end{document}